\documentclass[
final
]{dmtcs-episciences}

\usepackage[utf8]{inputenc}
\usepackage{subfigure}
\usepackage[usenames]{color}
\usepackage{amssymb}
\usepackage{amsmath}
\usepackage{amsthm}
\usepackage{amsfonts}
\usepackage{amscd}
\usepackage{graphicx}

\usepackage{color}
\usepackage{float}

\usepackage{graphics}
\usepackage{latexsym}
\usepackage{epsf}
\usepackage{breakurl}
\newcommand{\seqnum}[1]{\href{https://oeis.org/#1}{\rm \underline{#1}}}

\def \modd#1 #2{#1\ ({\rm mod}\ #2)}
\def\suchthat{\, : \,}
\DeclareMathOperator{\per}{per}
\DeclareMathOperator{\ce}{ce}

\theoremstyle{plain}
\newtheorem{theorem}{Theorem}

\newtheorem{lemma}[theorem]{Lemma}
\newtheorem{proposition}[theorem]{Proposition}

\theoremstyle{definition}

\newtheorem{observation}[theorem]{Observation}
\theoremstyle{remark}

\usepackage[round]{natbib}

\author[Lubomíra Dvořáková et al.]{Lubomíra Dvořáková \thanks{ORCID:0000-0001-7208-4248} \affiliationmark{1}
  \and Edita Pelantová \thanks{ORCID:0000-0003-3817-2943} \affiliationmark{1}
  \and Jeffrey Shallit \thanks{ORCID:0000-0003-1197-3820} \affiliationmark{2}\thanks{Research supported by NSERC grant RGPIN-2024-03725.}}
\title[On a sequence of Kimberling]{On a sequence of Kimberling and its relationship to the Tribonacci word}
\affiliation{
  FNSPE, Czech Technical University, Prague, Czech Republic\\
  School of Computer Science, University of Waterloo, Waterloo, Canada  }
\begin{document}

\publicationdata{vol. 28:2}{2026}{19}{10.46298/dmtcs.16926}{2025-11-14; 2025-11-14; 2026-03-06}{2026-03-09}
\maketitle

\begin{abstract} In 2017, Clark Kimberling defined an interesting  sequence ${\bf B} = 0100101100 \cdots$ of $0$'s and $1$'s by
certain inflation rules, and he made a number of conjectures about this
sequence and some related ones.  In this note we prove his conjectures
using, in part, the {\tt Walnut} theorem-prover.  We show how his word
is related to the infinite Tribonacci word, and we determine both the factor complexity and critical exponent of  $\bf B$.
\end{abstract}



\section{Introduction}
In June 2017, Clark Kimberling defined sequence \seqnum{A288462} in the 
OEIS \cite{oeis} as follows:  it is the infinite fixed point of the inflation
rules $00 \mapsto 0101$,
$1 \mapsto 10$, starting with $00$. 
Because these rules involve
a type of substitution more complicated than just a morphism, it is more
challenging to analyze.  In this note we prove his conjectures
using, in part, the {\tt Walnut} theorem-prover.  We also show how his word
is related to the infinite Tribonacci word, and we determine both the factor complexity and critical exponent of $\bf B$.

It may seem odd to devote an entire paper to a particular sequence, but the methods we use are widely applicable, and so the paper may serve as a primer on how to attack such a sequence using a combination of combinatorial and automata-theoretic techniques.  

As stated, Kimberling's original description is perhaps slightly vague,
so here is some elaboration.   We start
with $B_0 = 00$.  To find $B_{i+1}$ from $B_i$, we do the following:  we
factor $B_i$ into maximal blocks of the form $00$, $0$, and $1$; here maximal
means we cannot extend a block further to the right or left.  Then 
$B_{i+1}$ is the result of applying
the inflation rules $0 \mapsto 0$, $1 \mapsto 10$, and $00 \mapsto
0101$ to $B_i$.

For example, here are the first few iterates:
\begin{align*}
B_0 &= 00 \\
B_1 &= 0101 
\end{align*}
\begin{align*}
B_2 &= 010010 \\
B_3 &= 0100101100 \\
B_4 &= 010010110010100101 \\
B_5 &= 01001011001010010110010010110010 .
\end{align*}
It is not hard to see that $B_{i+1}$ is a prefix of $B_i$ for $i \geq 2$,
and so there is a unique infinite word ${\bf B} = a_0 a_1 a_2 \cdots =
0100101100\cdots$ 
of which all the $B_i$, $i \geq 2$, are prefixes.

Consider $|B_i|$, the length of the $i$'th iterate.
Let $c_0 = 2$, $c_1 = 4$, $c_2 = 6$, $c_3 = 10$, and
define $c_i = 2c_{i-1} - c_{i-4}$ for $i \geq 4$; this is
sequence \seqnum{A288465} in the OEIS.
Kimberling conjectured that $|B_i| = c_i$ for all $i \geq 0$.
In this note, we prove
Kimberling's conjecture, as well as his conjectures about the
related sequences \seqnum{A288463} and \seqnum{A288464}.
We also find other properties of the sequence ${\bf B}$ that
link it to the infinite Tribonacci word \seqnum{A080843}.

Kimberling indexed the sequence ${\bf B}$ starting at position $1$.
However, for using {\tt Walnut}, it is easier to index starting
at position $0$, and this is the convention we use in this paper
at the beginning.  Later on, in Section~\ref{related}, we will have to use his indexing.

\section{Kimberling's conjecture}

In this section we prove Kimberling's conjecture about the lengths
of the words $B_i$.

\begin{proposition}
Let $c_0 = 2$, $c_1 = 4$, $c_2 = 6$, $c_3 = 10$, and
define $c_i = 2c_{i-1} - c_{i-4}$ for $i \geq 4$.  Then $|B_i|= c_i$ for $i \geq 0$.
\end{proposition}

\begin{proof}
A simple induction now shows that, for $n \geq 1$, the word $B_i$ starts with $01$
and contains no occurrences of either $000$ or $111$.  Therefore $B_i$ can be
factorized uniquely as a concatenation of the single initial $0$ and blocks of the
form $1$, $10$, and $100$.  For $x \in \{ 1, 10, 100 \}$, define $N_x (i)$ to be
the number of occurrences of the block $x$ in this
factorization of $B_i$.  Then clearly
$|B_i| = 1 + N_1(i) + 2N_{10} (i) + 3N_{100} (i)$.

Define a mapping $\xi$ on the blocks $x$ that obeys Kimberling's inflation 
rules, and hence sends $1$ to $10$, $10$ to $100$, and
$100$ to $100101$.  Write $B_i = 0 x_1 x_2 \cdots x_k$, where each $x_j \in \{ 1, 10, 100 \}$, $1 \leq j \leq k$. By checking what happens at the boundaries, we see that $B_{i+1} = 0 \xi(x_1) \cdots \xi(x_k)$.  

We now claim that the following relations hold for $i \geq 2$:
\begin{align}
N_1 (i) &= N_{100} (i-1) \label{eqn1} \\
N_{10} (i) &= N_{100} (i-1) + N_1 (i-1) \label{eqn10} \\
N_{100} (i) &= N_{10} (i-1) + N_{100} (i-1). \label{eqn100}
\end{align}
These follow immediately by looking at the image of each block above.

We now show that these identities are enough to directly prove the following (no induction needed!):
\begin{align}
N_1 (i) &= N_1(i-1) + N_1(i-2) + N_1(i-3) \label{en1}\\
N_{10} (i) &= N_{10}(i-1) + N_{10}(i-2) + N_{10}(i-3) \label{n10} \\
N_{100} (i) &= N_{100}(i-1) + N_{100}(i-2) + N_{100} (i-3) \label{n100} 
\end{align}
for $i \geq 4$. 

Proof of Eq.~\eqref{en1}:
\begin{align*}
N_1 (i) &= N_{100} (i-1) \quad \text{(by Eq.~\eqref{eqn1})}\\
&= N_{10} (i-2) + N_{100}(i-2)  \quad \text{(by Eq.~\eqref{eqn100})}\\
&= N_{10} (i-2) + N_1(i-1)  \quad \text{(by Eq.~\eqref{eqn1})} \\
&= N_{100} (i-3) + N_1(i-3) + N_1(i-1) \quad \text{(by Eq.~\eqref{eqn100})} \\
&= N_1(i-2) + N_1(i-3) + N_1(i-1) \quad \text{(by Eq.~\eqref{eqn1})} .
\end{align*}

Proof of Eq.~\eqref{n10}:
\begin{align*}
N_{100} (i) &= N_{10} (i-1) + N_{100} (i-1) \quad \text{(by Eq.~\eqref{eqn100})} \\
&= N_{100} (i-2) + N_1 (i-2) + N_{100} (i-1) \quad \text{(by Eq.~\eqref{eqn10})} \\
&= N_{100} (i-2) + N_{100} (i-3) + N_{100} (i-1) \quad \text{(by Eq.~\eqref{eqn1})}.
\end{align*}

Proof of Eq.~\eqref{n100}:
\begin{align}
N_{10} (i) &= N_{100} (i-1) + N_1 (i-1) \quad \text{(by Eq.~\eqref{eqn10})}  \nonumber\\
&= N_{10} (i-2) + N_{100} (i-2) + N_1(i-1) \quad \text{(by Eq.~\eqref{eqn100})} \nonumber \\
&= N_{10} (i-2) + (N_{10} (i-1) - N_1 (i-2)) + N_1 (i-1) \quad \text{(by Eq.~\eqref{eqn10})} \label{done} .
\end{align}

We also have
\begin{align*}
N_1(i-1) - N_1 (i-2) &= N_{100} (i-2) - N_{100} (i-3) \quad \text{(by Eq.~\eqref{eqn1})} \\
&= N_{10} (i-3) \quad \text{(by Eq.~\eqref{eqn100})},
\end{align*}
and substituting into Eq.~\eqref{done} gives the desired result for $N_{10} (i)$.

Now from above we know that $|B_i| - 1 = N_1 (i) + 2 N_{10} (i) + 3 N_{100} (i)$.  Since each of the
sequences $(N_1 (i))_i$, $(N_{10} (i))_i$, and $(N_{100} (i))_i$ 
on the right-hand side is annihilated by the shift polynomial $X^3 - X^2 - X - 1$, so is their linear
combination  $(N_1 (i) + 2 N_{10} (i) + 3 N_{100} (i))_i$.  And since $X-1$ annihilates the
constant sequence $1$, the sequence $(|B_i|)_i$ is annihilated by the product of
$X-1$ and $X^3 - X^2 - X - 1$, which is $X^4 - 2X^3 + 1$.   In other words,
$|B_{i}| = 2|B_{i-3}| - |B_{i-4}|$ for $i \geq 4$.  After comparing the initial conditions
$i = 0,1,2,3$, it now follows that
$|B_i| = c_i$ for all $i \geq 0$.
\end{proof}

\section{Relationship to the Tribonacci word}

The infinite sequence $\bf B$ given in \seqnum{A288462} is
quite closely related to the celebrated Tribonacci word
${\bf TR} = t_0 t_1 t_2 \cdots = 01020100102010102010010201020100102010102010
\cdots$, the fixed point of the morphism $0 \mapsto 01$, $1 \mapsto 02$, $2 \mapsto 0$.  For more information about $\bf TR$,  see, for example, \cite{Chekhova&Hubert&Messaoudi:2001}.

We will need some additional concepts.
Fix an infinite word $\bf x$.  We say $w$ is a {\it return word\/}
to $y$ in $\bf x$ if
${\bf x}[i..i+n-1]$ and ${\bf x}[j..j+n-1]$ for $i < j$ are two consecutive occurrences of $y$ in $\bf x$ and
$w = {\bf x}[i..j-1]$.   If $y$ occurs with bounded gaps in $\bf x$, we can write $\bf x$ as a concatenation of a finite prefix $v$ and the $t$ different return words to $y$, and hence write ${\bf x} =v \pi({\bf z})$ 
for some morphism $\pi$ and $\bf z$ a word over $\{ 0,1,\ldots ,t-1\}$.  Then $\bf z$ is called the {\it derived sequence} of $y$ in $\bf x$ and is denoted by ${\bf d}_{\bf x}(y)$. 

\begin{theorem}\label{thm:main}
Kimberling's sequence ${\bf B}$ is equal to $0f(\bf TR)$, where $f: 0\mapsto 10, \ 1\mapsto 0, \ 2\mapsto 1$.
\label{description}
\end{theorem}
\begin{proof}
It is easy to see that the return words to $10$ in $\bf B$ are $100, \ 101, \ 10$.
If we code $100$ with the letter $0$, $101$ with the letter $1$ and $10$ with the letter $2$, we get the derived sequence 
$${\bf d}_{\bf B}(10)=0102010010201010201001020102\cdots$$
Moreover, by the definition of $\bf B$, the sequence $0^{-1}\bf B$ is fixed under the following inflation rules applied to the return words:  $100\mapsto 100101, \ 101\mapsto 10010, \ 10\mapsto 100$. This immediately implies that the derived sequence is fixed under the morphism $\varphi: \ 0\mapsto 01, \ 1\mapsto 02, \ 2\mapsto 0$; i.e., the derived sequence ${\bf d}_{\bf B}(10)$ is equal to $\bf TR$.  Consequently, ${\bf B}=0\pi(\bf TR)$, where $\pi: 0\mapsto 100, \ 1\mapsto 101, \ 2\mapsto 10$.  

One can check that $\pi = f\circ \varphi$.\footnote{We thank to Pascal Ochem who  pointed out to us that the morphism $\pi$ can be replaced by the simple morphism $f$.}   
Hence, ${\bf B}=0\pi({\bf TR}) = 0f\bigl(\varphi({\bf TR})\bigr)  =  0f(\bf TR)$ as stated.
\end{proof}

Our next goal is to find a finite automaton that computes the sequence $\bf B$.  For this, we need the notion of Tribonacci representation of an integer.

By a well-known theorem \cite{Carlitz&Scoville&Hoggatt:1972},
every integer $n \geq 0$ can be written uniquely as a sum of
distinct Tribonacci numbers $T_i$ for $i \geq 2$, provided
one never uses three consecutive Tribonacci numbers in the
representation.   If we write
$n = \sum_{2 \leq i \leq t} e_i T_i$, we can alternatively
represent $n$ by the binary word $e_t e_{t-1} \cdots e_2$.
For example, $17 = 13 + 4 = T_6 + T_4$, and its Tribonacci representation is therefore
$10100$.

Our automaton that computes $\bf B$ is a DFAO (deterministic finite automaton
with output) computing ${\bf B}[n]$.  The input is $n$,
expressed in Tribonacci representation, and the output is
${\bf B}[n]$.  See, for example, \cite{Mousavi&Shallit:2015}.

Let $c_i (n)$ denote the number of occurrences of the letter $i$ in the length-$n$ prefix of $\bf TR$.  In
\cite[\S 10.12]{Shallit:2023},
it is shown how to obtain synchronous automata computing the maps
$n \mapsto c_i (n)$. By ``synchronous'' we mean  that these automata
(called {\tt c0}, {\tt c1}, {\tt c2}) take two inputs in parallel, $n$ and $x$, in Tribonacci representation and accept if and only if $x = c_i (n)$.  
See 
\cite{Shallit:2021} for more about the notion of synchronous automata.

From the description in Theorem~\ref{description} that ${\bf B} = 0 \pi({\bf TR})$, we therefore get the following algorithm for computing ${\bf B}[n]$:
\begin{itemize}
\item If $n=0$ then ${\bf B}[n]= 0$.
\item Otherwise, find $y$ such that $3c_0 (y) + 3 c_1(y) + 2c_2(y) + 1 \leq m < 3c_0 (y+1) + 3 c_1 (y+1) + 2c_2(y+1) + 1 $.  This is the position of $\bf TR$ that gives rise to the $n$'th letter of $\bf B$ under the map $\pi$.
\item Let $t = n - ( 3c_0 (y) + 3 c_1(y) + 2c_2(y) + 1)$.  This is the relative position within the image of ${\bf TR}[y]$ under $\pi$.  If $t = 0$ it is the first letter, if $t = 1$ it is the second letter, and so forth.
\item Then ${\bf B}[n] = 1$ if $t = 0$, or if $t=2$ and ${\bf TR}[y] = 1$; otherwise ${\bf B}[n] = 0$.
\end{itemize}

We can now write {\tt Walnut} code that implements this algorithm.  For more about {\tt Walnut} and its use in combinatorics on words, see
\cite{Shallit:2023}.  The first 7 lines are taken from \cite[\S 10.12]{Shallit:2023}.
\begin{verbatim}
reg shift {0,1} {0,1} "([0,0]|[0,1][1,1]*[1,0])*":

def triba "?msd_trib (s=0&n=0) | Ex $shift(n-1,x) & s=x+1":
# position of n'th 0 in Tribonacci, starting at index 1
def tribb "?msd_trib (s=0&n=0) | Ex,y $shift(n-1,x) & 
   $shift(x,y) & s=y+2":
# position of n'th 1 in Tribonacci, starting at index 1
def tribc "?msd_trib (s=0&n=0) | Ex,y,z $shift(n-1,x) &
   $shift(x,y) & $shift(y,z) & s=z+4":
# position of n'th 2 in Tribonacci, starting at index 1
def c0 "?msd_trib Et,u $triba(s,t) & $triba(s+1,u) & t<=n & n<u":
def c1 "?msd_trib Et,u $tribb(s,t) & $tribb(s+1,u) & t<=n & n<u":
def c2 "?msd_trib Et,u $tribc(s,t) & $tribc(s+1,u) & t<=n & n<u":

def find_t_and_y "?msd_trib Eu,v,a0,a1,a2,b0,b1,b2 $c0(y,a0) &
   $c1(y,a1) & $c2(y,a2) & $c0(y+1,b0) & $c1(y+1,b1) & $c2(y+1,b2) &
   u=3*a0+3*a1+2*a2+1 & v=3*b0+3*b1+2*b2+1 & u<=n & n<v & t+u=n"::
# 26 states

def bb "?msd_trib (Et,y $find_t_and_y(n,t,y) & 
   ((t=0) | (t=2 & TR[y]=@1)))":
combine B bb:
\end{verbatim}
This gives us the automaton in Figure~\ref{baut}.

Now that we have the automaton for $\bf B$, we can use {\tt Walnut} to provide rigorous proofs of assertions about the sequence $\bf B$.  We only need to phrase our assertions in first-order logic, and {\tt Walnut} can decide if they are {\tt TRUE} or {\tt FALSE}.
As an example of the utility of the automaton for $\bf B$, we now use {\tt Walnut} to prove a result about the balance of $\bf B$.
A sequence over $\{0,1\}$
is said to be $k$-balanced if for all factors $x,y$ of the
same length, the number of $1$'s in $x$ differs from the number
of $1$'s in $y$ by at most $k$ \cite{Berstel:1996}.

\begin{theorem}
The sequence $\bf B$ is $3$-balanced but not $2$-balanced.
\end{theorem}

\begin{proof}
We can prove this with {\tt Walnut}.  We need an automaton computing {\tt bpref1}, the number of $1$'s in ${\bf B}[0..n-1]$.   We can compute this using the same technique that we used to construct the automaton $\tt B$.
\begin{verbatim}
def bpref1 "?msd_trib (n<=1 & z=0) | 
Et,y,x,a0,a1,a2 $find_t_and_y(n-1,t,y) & $c0(y,a0) & $c1(y,a1) &
   $c2(y,a2) & z=a0+2*a1+a2+x+1 & x<=1 & (x=1 <=> (t=2 & B[n-1]=@1))":
# z = the number of 1's in B[0..n-1]
def bfact1 "?msd_trib Ex,y $bpref1(i,x) & $bpref1(i+n,y) & z+x=y":
# z = the number of 1's in B[i..i+n-1]
eval bal3 "?msd_trib An,i,j,x,y ($bfact1(i,n,x) & $bfact1(j,n,y) & 
   x<=y)  => y<=x+3":
eval bal2 "?msd_trib An,i,j,x,y ($bfact1(i,n,x) & $bfact1(j,n,y) & 
   x<=y) => y<=x+2":
\end{verbatim}
The first returns {\tt TRUE} and the second {\tt FALSE}.  (It fails at
$n = 47$, as observed by Pierre Popoli.)
\end{proof}

\begin{figure}[htb]
\begin{center}
\includegraphics[width=5.5in]{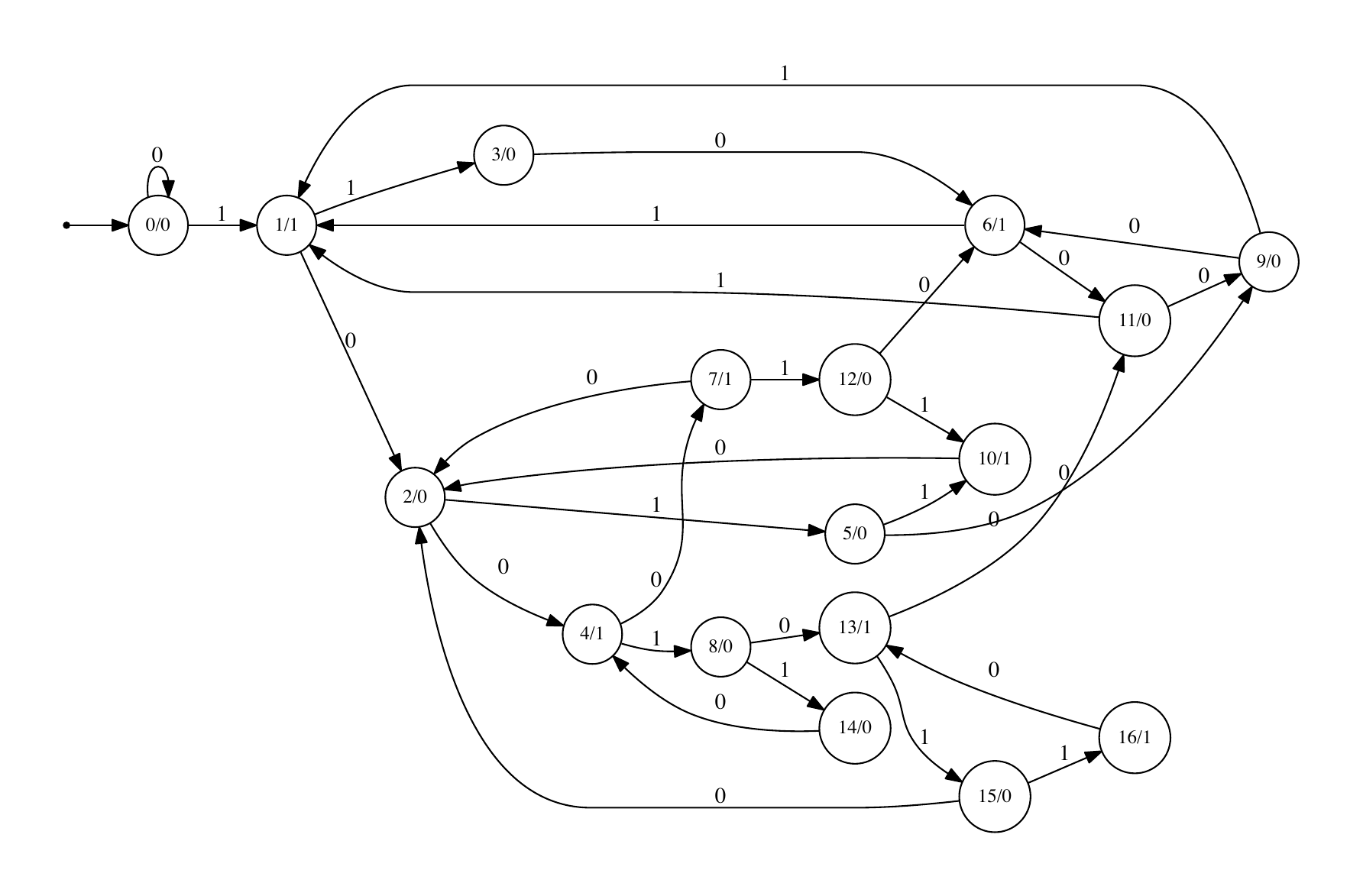}
\end{center}
\caption{The Tribonacci automaton for ${\bf B}[n]$.}
\label{baut}
\end{figure}

\section{Related sequences}
\label{related}

Kimberling also proposed the sequence \seqnum{A288464}, which 
consists of $I_1 (n)$, the index of the $n$'th $1$ in the sequence
$\bf B$, for $n \geq 1$.
However, he indexed $\bf B$ starting with $1$.   We keep his indexing here.
We can compute this with {\tt Walnut} as follows:
\begin{verbatim}
def nth1 "?msd_trib $bpref1(x,n) & $bpref1(x-1,n-1)":
# 69 states
\end{verbatim}

Similarly, we can compute $I_0 (n)$, the index of the $n$'th $0$ in
$\bf B$, again with $\bf B$ indexed starting at $1$.  This is sequence \seqnum{A288463} in the OEIS.
\begin{verbatim}
def bpref0 "?msd_trib Ey $bpref1(n,y) & z+y=n":
def nth0 "?msd_trib $bpref0(x,n) & $bpref0(x-1,n-1)":
# 30 states
\end{verbatim}

Kimberling conjectured that 
$$ -1 < \psi  - I_0(n)/n < 1$$
for $n \geq 1$ and some constant $\psi \doteq 1.83$.  It turns out
that $\psi = 1.8392867552\cdots$ is the Tribonacci constant, the unique real zero of the polynomial
$X^3-X^2-X-1$.

Kimberling also conjectured that
$$ -1 < \gamma - I_1(n)/n < 1$$
for $n \geq 1$ and some constant $\gamma \doteq 2.19$.  It turns out
that $\gamma = (\psi^2 + 1)/2 = 2.19148788\cdots $.

We will now prove more precise versions of these claims.
\begin{theorem}
For $n \geq 1$ we have
\begin{itemize}
\item[(a)] $\lfloor \psi n \rfloor - 2 \leq I_0(n)  \leq 
\lfloor \psi n \rfloor + 2$;
\item[(b)] $\lfloor \gamma n \rfloor -1 \leq 
I_1 (n) \leq \lfloor \gamma n \rfloor + 3$.
\end{itemize}
\end{theorem}

\begin{proof}
\leavevmode
\begin{itemize}
\item[(a)] We use an estimate from \cite[Eq.~(30)]{Dekking&Shallit&Sloane:2020}; namely
\begin{equation}
\lfloor \psi n \rfloor - 1 \leq A_0(n) \leq 
\lfloor \psi n \rfloor + 1 
\label{an}
\end{equation}
for $n \geq 1$, where $A_0(n)$ is the position of the $n$'th $0$
in $\bf TR$, where $\bf TR$ is also indexed starting at
position $1$.   (Similar estimates can be found in \cite{Richomme&Saari&Zamboni:2010}.)

We also use some {\tt Walnut} code from 
\cite{Shallit:2023} for $A_0(n)$, namely the 
automaton {\tt triba}.

Now we show that $-1 \leq A_0(n) - I_0(n) \leq 1$:
\begin{verbatim}
eval cmp "?msd_trib An,x,y ($triba(n,x) & $nth0(n,y)) =>
   (x=y+1|y=x+1|x=y)":
\end{verbatim}
And {\tt Walnut} returns {\tt TRUE}.

Putting this together with Eq.~\eqref{an}, we get the estimate
$$ -2 \leq \lfloor \psi n \rfloor - I_0(n) \leq 2,$$
from which Kimberling's first inequality follows easily.

\item[(b)]  Let $A_1(n)$ denote the position of the $n$'th occurrence of $1$ in the
Tribonacci word $\bf TR$ (indexed starting at $1$).  In
\cite[Eq.~(31)]{Dekking&Shallit&Sloane:2020} the authors showed
$ \lfloor \psi^2 n \rfloor - 2 \leq A_1(n) \leq \lfloor \psi^2 n \rfloor + 1$
from which we get
\begin{equation}
-1 \leq \frac{A_1(n) - \lfloor \psi^2 n \rfloor}{2} \leq 1/2
\label{beq}
\end{equation}
by
rearranging.

On the other hand, we can use {\tt Walnut} to prove that
$$A_1(n) \leq 2 I_1 (n) + 1-n \leq A_1(n) + 5.$$
For $A_1(n)$ we use the code {\tt tribb} from \cite[\S 10.12]{Shallit:2023}:
\begin{verbatim}
eval comp2 "?msd_trib An,x,y,z ($tribb(n,x) & $nth1(n,y) &
   z+n=2*y+1) => (x<=z & z<=x+5)":
\end{verbatim}
And {\tt Walnut} returns {\tt TRUE}.
From this we get
$$ 0 \leq \frac{2 I_1 (n) - n + 1 - A_1(n)}{2} \leq 5/2.$$
Adding this to Eq.~\eqref{beq}
we get
$$-1 \leq  \frac{2 I_1 (n) - n - \lfloor \psi^2 n \rfloor + 1}{2} \leq 3$$
which, by rearranging, implies that
$$ -3/2 \leq I_1 (n) - \lfloor \gamma n \rfloor \leq 3$$
for $\gamma = (\psi^2+1)/2$.  Since $I_1 (n) - \lfloor \gamma n \rfloor$
is an integer, we get
$$ -1 \leq I_1 (n) - \lfloor \gamma n \rfloor \leq 2,$$
from which
Kimberling's second inequality now follows easily.
\end{itemize}
This completes the proof.
\end{proof}

\section{Factor complexity and critical exponent}

Recall that by {\it factor\/} we mean  a contiguous block of letters within a word.

Recall that the factor complexity (aka factor complexity) of a sequence is the function mapping $n$ to the number of distinct blocks of length $n$ appearing in it.  A word is called a Rote word, defined in~\cite{Rote:1994}, if its factor complexity is $2n$ for $n \geq 1$.

We will also need the notation of critical exponent of an infinite word.
We say a finite word $w = w[1..n]$ has period $p$
if $w[i]=w[i+p]$ for all $i$, $1 \leq i \leq n-p$.  The smallest nonzero period is called {\it the\/} period and is denoted $\per(w)$.  The exponent of a finite nonempty word $w$ is then
$\exp(w) := |w|/\per(w)$.  Let $\bf x$ be an infinite word.  Then the critical exponent of $\bf x$, written $\ce({\bf x})$, is $\sup \{ \exp(w) \suchthat
w \text{ is a factor of {\bf x}} \}$.

The goal of this section is to prove that the sequence $\bf B$ belongs to the class of Rote words and has the same critical exponent as that of the Tribonacci word, see~\cite{Tan&Wen:2007}. 

More specifically, we aim to prove the following two theorems.

\begin{theorem}\label{thm: complexity}
The factor complexity of $\bf B$ is $2n$ for $n \geq 1$.
\end{theorem}

\begin{theorem}\label{thm: critical_exponent}
The critical exponent of $\bf B$ is $2+\frac{1}{\psi-1} = 3.19148788395\cdots$,
where $\psi$ is the real zero of $X^3-X^2-X-1$.
\end{theorem}

In principle, {\tt Walnut} could be used to prove both of these theorems; but in practice, we were unable to complete the proof because the computations fail to terminate within reasonable bounds on space and time.  So in the next section, we use some known theory instead.

\subsection{Bispecial factors of the sequence $\bf B$}

For a binary sequence $\bf x$, we say a factor $w$ is right-special if $w0$ and $w1$ both appear in $\bf x$, and left-special if both $0w$ and $1w$ both appear in $\bf x$.  If $w$ is both right- and left-special, we say it is {\it bispecial}.

In order to determine both the factor complexity and the critical exponent of $\bf B$, the knowledge of bispecial factors in $\bf B$ is essential.
The description of bispecial factors and their return words in $\bf TR$ is taken from~\cite{Droubay&Justin&Pirillo:2001, Glen:2007, Dvorakova&Pelantova:2024}.
The sequence $(b_n)_{n=0}^{\infty}$ of all non-empty bispecial factors, ordered by length, in the Tribonacci word, satisfies $b_0=0$ and for $n\geq 1$,
$$b_n=\varphi(b_{n-1})0,\quad \text{where}\ \varphi: 0\mapsto 01, \ 1\mapsto 02, \ 2\mapsto 0\,.$$
Moreover, if $i \equiv \modd{n} {3}$ for $i\in\{-1,0,1\}$, then the both-sided extensions of $b_n$ are 
\begin{equation}\label{eq: bext}
(i+1)b_nj,\ jb_n(i+1)\quad \text{for}\ j\in\{0,1,2\}\,.
\end{equation}

By Theorem \ref{thm:main},    Kimberling'sequence is equal to $f({\bf TR})$. As ${\bf TR}$ is fixed by $\varphi$,  ${\bf B}$ is also the image  of the Tribonacci sequence under the morphism $\pi = f\circ \varphi$. Instead of $f$ we will work with the morphism  $\pi: 0\mapsto 100, \ 1\mapsto 101, \ 2\mapsto 10$ because it allows us to derive the form of bispecial factors in $\bf B$ in an easier way.

\begin{observation}
The non-empty bispecial factors of $\bf B$ of length $\leq 4$ are $0,1,01,10,010$.
\end{observation}

\begin{lemma}\label{lem: BS}
The complete list of bispecial factors of $\bf B$ of length $\geq 5$ is as follows:
\begin{enumerate}
\item for $n \equiv \modd{0} {3}$:
$$\pi(b_n)10, \ \pi(b_n)101\,;$$
\item for $n\equiv \modd{1} {3}$:
$$\pi(b_n)10, \ \pi(b_n)101, \ 0\pi(b_n)10, \ 0\pi(b_n)101\,;$$
\item for $n\equiv \modd{2} {3}$:
$$\pi(b_n)10, \ 0\pi(b_n)10\,.$$
\end{enumerate}
\end{lemma}
\begin{proof}
On the one hand, using~\eqref{eq: bext}, it follows that the factors from Items 1 to 3 are all of the bispecial factors in $\bf B$ obtained by applying $\pi$ to both-sided extensions of $b_n$. For instance, for $n\equiv \modd{0} {3}$, the both-sided extensions of $b_n$ are $1b_n0,\ 1b_n1,\ 1b_n2,\ 0b_n1$, and $ 2b_n1$. 
Since 
\begin{align*}
\pi(1b_n0) &=101
\pi(b_n)100,\\
\pi(0b_n10)&=100\pi(b_n)101100,\\ \pi(1b_n20) &=101\pi(b_n)10100
\end{align*}
are factors of $\bf B$, it follows that $\pi(b_n)10$ and $\pi(b_n)101$ are bispecial factors in $\bf B$. 
On the other hand, each bispecial factor $w$ in $\bf B$ of length at least 5 starts with $10$ or $010$ and ends with $10$ or $101$. By the form of the morphism $\pi$, the factor $w$ takes one of the following forms
$$w\in\{\pi(b)10, \ \pi(b)101, \ 0\pi(b)10,\ 0\pi(b)101\}\,,$$
where $b$ is a~non-empty bispecial factor in $\bf TR$. Consequently, the factor $w$ is included in the list from Lemma~\ref{lem: BS}. 
\end{proof}

\subsection{Factor complexity of the sequence $\bf B$}
\begin{lemma}\label{lem: LS}
The set of left special factors of $\bf B$ is equal to the set of prefixes of $\pi(\bf TR)$ and of $\bf B=0\pi(\bf TR)$. 
\end{lemma}
\begin{proof}
Since $\bf B$ is aperiodic, each left special factor is the prefix of a bispecial factor. 
All bispecial factors of $\bf B$ are prefixes of $\bf B$ or $\pi(\bf TR)$. This statement is clear for bispecial factors of length $\leq 4$. 
It remains to see that the bispecial factors of length $\geq 5$, as listed in Lemma~\ref{lem: BS}, are prefixes of $\bf B$ or $\pi(\bf TR)$.
It can be easily proven by induction that $b_n(i+1)$ is the prefix of $\bf TR$ for $n\equiv \modd{i} {3}$, where $i\in\{-1,0,1\}$.
Therefore, for instance for $n\equiv \modd{0} {3}$, the bispecial factors $\pi(b_n)10$ and $\pi(b_n)101$ are prefixes of $\pi(b_n1)$, which is a~prefix of $\pi(\bf TR)$.
We can proceed analogously for $n\equiv \modd{1} {3}$ and $n\equiv \modd{2} {3}$.
\end{proof}

\begin{proof}[of Theorem~\ref{thm: complexity}]
By Lemma~\ref{lem: LS}, for each length $n\geq 1$, there are two distinct left special factors in $\bf B$, which confirms that the factor complexity is $2n$ for all $n\geq 1$.
\end{proof}

\subsection{Return words to bispecial factors in $\bf B$}
For the purpose of computing the critical exponent of $\bf B$, we intend to apply the following theorem. 

\begin{theorem}[\cite{Dolce&Dvorakova&Pelantova:2023}, Theorem 3]\label{thm:FormulaForE}
Let $\bf u$ be a uniformly recurrent aperiodic sequence.
Let $(w_n)_{n\in\mathbb N}$ be the sequence of all bispecial factors in $\bf u$ ordered by length.
For every $n \in \mathbb N$, let $v_n$ be the shortest return word to the bispecial factor $w_n$ in $\bf u$.
Then
$$
\ce({\bf u}) = 1 + \sup\left\{\frac{|w_n|}{|v_n|}\ : \ n\in \mathbb N \right\}\,.
$$
\end{theorem}

It is thus necessary to describe the shortest return words to bispecial factors in $\bf B$.

There are three return words to each factor in the Tribonacci word. In particular, the return words to $b_n$, for $n\geq 0$, are 
$$\varphi^n(0),\ \varphi^n(01),\ \varphi^n(02)\,.$$

\begin{proposition} \label{prop: returns}
Let $w$ be a bispecial factor from the list in Lemma~\ref{lem: BS}.
\begin{itemize}
\item[(a)] If $w=\pi(b_n)10$, then the shortest return word to $w$ equals $\pi(\varphi^n(0))$;
\item[(b)] If $w =\pi(b_n)101$ or $b=0\pi(b_n)10$, then each return word to $w$ has length $\geq |\pi(\varphi^n(0))|$;
\item[(c)] If $w=0\pi(b_n)101$, then each return word to $w$ has length $\geq |\pi(\varphi^n(0))|+|\pi(\varphi^{n-1}(0))|$.
\end{itemize}
\end{proposition}
\begin{proof}
Using the form of $\pi$ and its injectivity, the factor $\pi(b_n)$ has a~unique preimage $b_n$, therefore the shortest complete return word $\varphi^n(0)b_n$ to $b_n$ gives rise to the shortest complete return word $\pi(\varphi^n(0)b_n)$ to $\pi(b_n)$. 
\begin{itemize}
\item[(a)] Since $\pi(b_n)$ is always followed by $10$, the factor $\pi(b_n)10$ has the same shortest return word as $\pi(b_n)$. This proves (a).
\item[(b)] The claim (b) follows immediately from (a).
\item[(c)] By Lemma~\ref{lem: BS}, the factor $w=0\pi(b_n)101$ is bispecial only for $n\equiv \modd{1} {3}$. Since the last letter of $\varphi^n(0)$ equals $i$, where $n\equiv \modd{i} {3}$, the factor $0\pi(\varphi^n(0)b_n)101$ has the suffix $101\pi(b_n)101$, which proves that $0\pi(\varphi^n(0))0^{-1}$ is not a~return word to $0\pi(b_n)101$. 

Since $0\pi(b_n)101$ contains $\pi(b_n)$, for each of its return words $v$, the word $0^{-1}v0$ is obtained as a~concatenation of return words to $\pi(b_n)$. This concatenation is not equal to $\pi(\varphi^n(0))$, and hence $$\begin{array}{rcl}
|v|&\geq & \min\{|\pi(\varphi^n(00))|, |\pi(\varphi^n(01))|, |\pi(\varphi^n(02))|\}=\\
&=&|\pi(\varphi^n(02))|=|\pi(\varphi^n(0))|+|\pi(\varphi^{n-1}(0))|.\end{array}$$ 
\end{itemize}
This concludes the proof.
\end{proof}

\subsection{Critical exponent of the sequence $\bf B$}
In order to apply Theorem~\ref{thm:FormulaForE}, we need to determine the lengths of bispecial factors and their shortest return words in $\bf B$.

Recall that in $\bf TR$, the sequence $(b_n)_{n=0}^{\infty}$ of all non-empty bispecial factors satisfies $b_0=0$ and 
$b_n=\varphi(b_{n-1})0$ for $n\geq 1$ and $r_n=\varphi^n(0)$ is the shortest return word to $b_n$.
Also recall that the Parikh vector of a word $x$ over the alphabet $\{0,1,\ldots, t-1\}$
is the vector
$(|x|_0, |x|_1, \ldots, |x|_{t-1}),$
where $|x|_a$ is the number of occurrences of $a$ in $x$.

The Parikh vectors of bispecial factors and their shortest return words in $\bf TR$ are
\begin{equation}\label{eq: Parikh}
\vec b_n=\frac{1}{2}\left(\begin{array}{c} 
T_{n+3}+T_{n+1}-1\\
T_{n+2}+T_{n\phantom{-1}}-1\\
T_{n+1}+T_{n-1}-1
\end{array}\right) \quad \text{and}\quad \vec r_n=\left(\begin{array}{c} T_{n+1}\\ T_{n\phantom{-1}} \\ T_{n-1}\end{array}\right) . 
\end{equation}
The explicit form of $T_n$ reads, for $n\geq 0$,
\begin{equation}\label{eq: explicit}
T_n=c_1 \psi_1^n+c_2\psi_2^n+c_3\psi_3^n\,,
\end{equation}
where $$\psi_1=\psi\doteq 1.8393, \ \psi_2=\overline{\psi_3}\doteq -0.4196+0.6063i,\ \text{and} \ c_j=\frac{1}{-\psi_j^2+4\psi_j-1}\ \text{for} j\in\{1,2,3\}\,.$$

The following lemma enables to express the lengths of all bispecial factors and their shortest return words in $\bf B$ in terms of the Tribonacci numbers.
\begin{lemma}\label{lem: lengthBS} 
For $n\geq 0$ we have
$$|\pi(r_n)|=T_{n+5}-T_{n+4} \quad \text{and} \quad |\pi(b_n)|=|\pi(r_n)|+T_{n+4}-4\,.$$
\end{lemma}
\begin{proof}
Using the Tribonacci recurrence, we get
$$|\pi(r_n)|=(1,1)M_\pi\vec r_n=(1,1)\left(\begin{array}{rcl} 2&1&1\\ 1&2&1\end{array}\right)\left(\begin{array}{c} T_{n+1}\\ T_{n\phantom{-1}} \\ T_{n-1}\end{array}\right)=T_{n+5}-T_{n+4}\,.$$
$$|\pi(b_n)|=(1,1)M_\pi\vec b_n=(1,1)\left(\begin{array}{rcl} 2&1&1\\ 1&2&1\end{array}\right)\frac{1}{2}\left(\begin{array}{c} 
T_{n+3}+T_{n+1}-1\\
T_{n+2}+T_{n\phantom{-1}}-1\\
T_{n+1}+T_{n-1}-1
\end{array}\right)=T_{n+5}-4\,.$$
This proves the assertion.
\end{proof}

\begin{proof}[of Theorem~\ref{thm: critical_exponent}]
Combining Theorem~\ref{thm:FormulaForE}, Lemma~\ref{lem: BS}, Proposition~\ref{prop: returns}, and Lemma~\ref{lem: lengthBS}, we have
 $$\ce({\bf B})\geq 1+\lim_{n\to\infty}\frac{|\pi(b_n)10|}{|\pi(r_n)|}=2+\lim_{n\to \infty}\frac{T_{n+4}-2}{T_{n+5}-T_{n+4}}=2+\frac{1}{\psi-1} . $$

 Now, for every bispecial factor $w$ of length $\leq 4$ in $\bf B$ and its shortest return word $v$, the inequality $1+\frac{|w|}{|v|}\leq 2+\frac{1}{\psi-1}\doteq 3.19$ holds.
Here is a table of such bispecial factors and their shortest return words.
$$\begin{array}{c|c|c|c|c|c}
    w &  0 & 1 & 01 & 10 & 010 \\ \hline
    v & 0 &  1 & 01 & 10 & 01 \\ \hline
    |w|/|v| & 1 & 1 & 1 & 1& 1.5 
\end{array}$$

To complete the proof that $\ce({\bf B})\leq 2+\frac{1}{\psi-1}=1+\frac{\psi}{\psi-1}$, it suffices to show that for all $n\geq 0$  we have
\begin{equation}\label{eq: inequality}
\frac{|\pi(b_n)101|}{|\pi(r_n)|}\leq \frac{\psi}{\psi-1}\quad \text{and} \quad \frac{|0\pi(b_n)101|}{|\pi(r_n)|+|\pi(r_{n-1})|}\leq \frac{\psi}{\psi-1}\,.
\end{equation}

The first inequality from~\eqref{eq: inequality} may be simplified as follows:
\begin{equation}\label{eq:theFirst}
\frac{T_{n+5}-1}{T_{n+5}-T_{n+4}}\leq \frac{\psi}{\psi-1}\,,\quad \text{or equivalently }\quad \psi  \leq  \frac{T_{n+5}-1}{T_{n+4}-1}\,.
\end{equation}

The second inequality from~\eqref{eq: inequality} can be rewritten as 
$$
\frac{T_{n+5}}{T_{n+5}-T_{n+3}}\leq \frac{\psi}{\psi-1}\,, \quad \text{or equivalently } \quad  
\psi \leq  \frac{T_{n+5}}{T_{n+3}}\,.
$$
Since $\frac{T_{n+5}}{T_{n+3}}\geq \frac{T_{n+5}-1}{T_{n+3}} \geq \frac{T_{n+5}-1}{T_{n+4}-1} $, only the first inequality \eqref{eq:theFirst} needs to be verified. It obviously holds for $n\in \{0,1\}$.
Using the explicit formula  \eqref{eq: explicit} for  $T_n$, we see that for $n \geq 2$ we have
$$T_n - c_1 \psi^n \in (-K, K),\ \  \text{where }K  \leq  2|c_2\psi^2_2| .
$$
For  $\psi\doteq 1.8393 $, 
$c_2\doteq -0.1681+0.1983i$ and $\psi_2\doteq -0.4196+ 0.6063i$, the parameter $K$ satisfies $K \leq 0.29$, and 
the inequality  $\psi\geq \frac{K+1}{1-K}$ holds. Hence 
  $$ \frac{T_{n+5}-1}{T_{n+4} - 1}\geq \frac{c_1\psi^{n+5} -1-K}{c_1\psi^{n+4} - 1 + K} \geq \psi .$$ 
The proof is now complete.
\end{proof}

\section{Acknowledgments}

We acknowledge with thanks conversations with Pierre Popoli.

\nocite{*}
\bibliographystyle{abbrvnat}
\bibliography{biblio-dmtcs}
\label{sec:biblio}

\end{document}